\documentclass[12pt]{amsart}

\usepackage{amsmath,amsthm,amssymb}
\usepackage[top=30truemm,bottom=30truemm,left=30truemm,right=30truemm]{geometry}
\usepackage[all]{xy}

\usepackage[dvipsnames]{xcolor}
\usepackage{iftex}
\ifPDFTeX
  \usepackage[linktocpage]{hyperref}
\else
  \usepackage[dvipdfmx,linktocpage]{hyperref}
\fi

\definecolor{mycolor}{rgb}{1.0,0,0.2}
\hypersetup{
    colorlinks=true,
    citecolor=mycolor,
    linkcolor=mycolor,
    linktoc=page,
}

\usepackage{aliascnt} 

\newtheorem{theorem}{Theorem}[section]

\newaliascnt{proposition}{theorem}
\newtheorem{proposition}[proposition]{Proposition}
\aliascntresetthe{proposition}

\newaliascnt{lemma}{theorem}

\aliascntresetthe{lemma}

\newaliascnt{corollary}{theorem}
\newtheorem{corollary}[corollary]{Corollary}
\aliascntresetthe{corollary}

\newaliascnt{definition}{theorem}
\newtheorem{definition}[definition]{Definition}
\aliascntresetthe{definition}

\newaliascnt{remark}{theorem}
\newtheorem{remark}[remark]{Remark}
\aliascntresetthe{remark}

\newaliascnt{example}{theorem}
\newtheorem{example}[example]{Example}
\aliascntresetthe{example}

\newcommand{\im}{\textrm{im}}

\newcommand{\End}{$\hfill{\blacksquare}$}

\newcommand{\gv}{{\rm GV}}

\newcommand{\prarrow}[2]{\ar@<0.5ex>[r]^-{#1} \ar@<-0.5ex>[r]_-{#2}}
\newcommand{\plarrow}[2]{\ar@<0.5ex>[l]^-{#1} \ar@<-0.5ex>[l]_-{#2}} 
\newcommand{\pdarrow}[2]{\ar@<0.5ex>[d]^-{#1} \ar@<-0.5ex>[d]_-{#2}} 
\newcommand{\puarrow}[2]{\ar@<0.5ex>[u]^-{#1} \ar@<-0.5ex>[u]_-{#2}}

\makeatletter
\newcommand{\subscripts}[3]{%
  \@mathmeasure\z@\displaystyle{#2}%
  \global\setbox\@ne\vbox to\ht\z@{}\dp\@ne\dp\z@
  \setbox\tw@\box\@ne
  \@mathmeasure4\displaystyle{\copy\tw@_{#1}}%
  \@mathmeasure6\displaystyle{{#2}_{#3}}%
  \dimen@-\wd6 \advance\dimen@\wd4 \advance\dimen@\wd\z@
  \hbox to\dimen@{}\mathop{\kern-\dimen@\box4\box6}%
}
\makeatother

\begin{document}

\title{Godbillon--Vey classes of Lie subalgebroids}
\author{Shuhei Yonehara}
\address{National Institute of Technology, Yonago College, Tottori, 683-8502, JAPAN}
\email{shuhei.yonehara.0201@gmail.com}

\begin{abstract}
  The Godbillon--Vey class is a secondary characteristic class which is defined for regular foliations and has been studied extensively. In this paper, we fix a Lie algebroid and relatively define the Godbillon--Vey class for its Lie subalgebroids, and study their properties. We also present several examples.
\end{abstract}

\maketitle

\tableofcontents


\section{Introduction}

The \textit{Godbillon--Vey class} was introduced in \cite{godbillon1971invariant} and has since been studied in various contexts in geometry and topology. In this paper, we focus on \textit{Lie algebroids}, which generalize
the tangent bundle, and construct Godbillon--Vey classes for Lie
subalgebroids. A Lie algebroid $\mathcal{A}\to M$ with anchor map $\rho:\mathcal{A}\to TM$ induces a \emph{singular foliation} $\rho(\Gamma_c(\mathcal{A}))$ on $M$. A singular foliation on a manifold $M$ is a $C^\infty(M)$-submodule $\mathcal{V}\subset \mathfrak{X}_c(M)$ ($\mathfrak{X}_c(M)$ denotes the set of all compactly supported vector fields on $M$) which is locally finitely generated and satisfies $[\mathcal{V},\mathcal{V}]\subset\mathcal{V}$.

The Godbillon--Vey class is a special case of secondary characteristic classes \cite{bott1972lectures}. Secondary characteristic classes for Lie algebroids have been studied by Crainic and Fernandes \cite{crainic2005secondary}, who constructed intrinsic characteristic classes associated with Lie algebroids themselves. In another direction, Balcerzak and Kubarski \cite{balcerzak2012exotic} defined secondary characteristic classes for pairs consisting of a Lie algebroid $\mathcal{A}$ and a Lie subalgebroid $\mathcal{B}\subset\mathcal{A}$ that induce the same regular foliation. Vaisman \cite{vaisman2010foliated} defined secondary characteristic
classes for Lie subalgebroids $\mathcal{B}\subset\mathcal{A}$ such that $\rho(\Gamma_c(\mathcal{B}))$ defines a regular foliation and $\mathcal{A}/\mathcal{B}$ locally descends to a Lie algebroid on a transversal. These constructions are related to the present work, but their settings differ from ours.

The purpose of the present paper is to construct Godbillon--Vey
classes for coorientable wide Lie subalgebroids, without assuming
regularity of their anchor images or the projectability condition
described above. More precisely, for a fixed Lie algebroid $\mathcal{A}$, we define the Godbillon--Vey class
$\gv_{\mathcal{A}}(\mathcal{B})\in H^{2q+1}(\mathcal{A})$
of a coorientable wide Lie subalgebroid $\mathcal{B}$ of corank $q$,
relative to $\mathcal{A}$. Moreover, we prove the naturality theorem:
\begin{theorem}(\autoref{thm:100}) 
  Let $\mathcal{A}\to M$ be a Lie algebroid, $\mathcal{B}\subset\mathcal{A}$ a (coorientable and wide) Lie subalgebroid and $f:N\to M$ a smooth map. Suppose that $f$ is transverse to $\mathcal{B}$. Then 
  \[\gv_{f^!\mathcal{A}}(f^!\mathcal{B})=f^!\gv_\mathcal{A}(\mathcal{B})\]
  holds.\hfill{$\Box$}
\end{theorem}

Extending the Godbillon--Vey class to singular foliations is a difficult problem (for an extension under suitable assumptions, see for example \cite{macdonald2021chern}). Although the present paper does not define characteristic classes for singular foliations themselves, it aims to provide a starting point for such a theory.

\vspace{0.2in}
\noindent{\bf Acknowledgments.} The author is grateful to Ryushi Goto for providing valuable insights into this paper. He would also like to thank Noriaki Ikeda, Hsuan-Yi Liao, and Ping Xu for their valuable comments on this work.

\vspace{0.2in}
\noindent{\bf AI disclosure.}
The author used ChatGPT (OpenAI) to assist with language editing and to identify potential gaps in proofs. All mathematical statements and proofs were independently verified by the author, who takes full responsibility for the content of this paper.

\section{Lie algebroid and Lie algebroid cohomology}\label{sec:Lie algebroid and Lie algebroid cohomology}

Throughout the paper, we denote by $M,N$ smooth manifolds. A Lie algebroid is a generalization of the tangent bundle obtained by twisting it via a vector bundle morphism called the anchor map. In this section, we review the basic notions and examples used in this paper. 

\begin{definition}
  A {\rm Lie algebroid} on $M$ is a triplet $(\mathcal{A},\rho,[\cdot,\cdot]_\mathcal{A})$ of a vector bundle $\mathcal{A}\to M$ and a bundle map $\rho:\mathcal{A}\to TM$, and a Lie bracket $[\cdot,\cdot]_\mathcal{A}$ on the space of sections $\Gamma(\mathcal{A})$, which satisfies the Leibniz identity
  \[[v,fw]_\mathcal{A}=f[v,w]_\mathcal{A}+(\mathcal{L}_{\rho(v)}f)w,\qquad v,w\in\Gamma(\mathcal{A}),\ f\in C^\infty(M).\]The map $\rho$ is called the {\rm anchor map}.\End
\end{definition}

For a Lie algebroid $(\mathcal{A},\rho,[\cdot,\cdot]_\mathcal{A})$,
\[\rho([v,w]_\mathcal{A})=[\rho(v),\rho(w)]\]holds for any $v,w\in\Gamma(\mathcal{A})$. Consequently, a Lie algebroid $(\mathcal{A},\rho,[\cdot,\cdot]_\mathcal{A})$ induces a singular foliation $\rho(\Gamma_c(\mathcal{A}))\subset \mathfrak{X}_c(M)$, where $\Gamma_c(\mathcal{A})$ denotes the set of compactly supported sections of $\mathcal{A}$.

\vspace{0.1in}
We present several examples of Lie algebroids below.

\begin{example}
  $(TM,\rm{id}_{TM},[\cdot,\cdot])$ is a Lie algebroid on $M$, and all Lie algebroids can be seen as generalizations of this example.\End
\end{example}

\begin{example}
  A finite dimensional Lie algebra defines a Lie algebroid over a point, whose anchor map is trivial.\End
\end{example}

\begin{example}[Action algebroids]
  Let $\varphi:\mathfrak{g}\to\mathfrak{X}(M)$ be a Lie algebra action of a Lie algebra $\mathfrak{g}$ on $M$. Then one can define a Lie algebroid structure on the trivial vector bundle $M\times\mathfrak{g}\to M$ as follows. The anchor map $\rho:M\times\mathfrak{g}\to TM$ is defined by
\[
\rho(x,v)=\varphi(v)_x.
\]
Identifying $\Gamma(\mathcal{A})$ with $C^\infty(M;\mathfrak{g})$, the Lie bracket on $\Gamma(\mathcal{A})$ is defined by
\[
[f,g](x)
=[f(x),g(x)]_{\mathfrak{g}}
+(\mathcal{L}_{\varphi(f(x))}g)(x)
-(\mathcal{L}_{\varphi(g(x))}f)(x).
\]
We denote this Lie algebroid by $M\ltimes\mathfrak{g}$.\End
\end{example}

\begin{example}[Lie algebroids induced by Poisson manifolds]\label{ex:0}

A Poisson structure on a manifold $M$ is a bivector $\pi\in\mathfrak{X}^2(M)$ which satisfies $[\pi,\pi]_S=0$, where $[\cdot,\cdot]_S$ denotes the Schouten bracket, which is defined by
\[[X,Y]_S=\displaystyle\sum_{i,j}(-1)^{i+j}[X_i,Y_j]\wedge X_1\wedge\cdots\wedge \hat{X_{i}}\wedge\cdots\wedge X_k\wedge Y_1\wedge\cdots\wedge \hat{Y_{j}}\wedge\cdots\wedge Y_l\]
for $X=X_1\wedge\cdots\wedge X_k$ and $Y=Y_1\wedge\cdots\wedge Y_l$. For a bivector $\pi\in\mathfrak{X}^2(M)$, $[\pi,\pi]_S\in\mathfrak{X}^3(M)$ is a 3-vector whose coefficients of local representation are given by
    \[([\pi,\pi]_S)_{ijk}=\displaystyle\sum_{l}\left(\displaystyle\frac{\partial\pi_{ij}}{\partial x_l}\pi_{lk}+\displaystyle\frac{\partial\pi_{ki}}{\partial x_l}\pi_{lj}+\displaystyle\frac{\partial\pi_{jk}}{\partial x_l}\pi_{li}\right).\]A pair $(M,\pi)$ of a manifold $M$ and a Poisson structure $\pi$ on $M$ is called a Poisson manifold.

Let $(M,\pi)$ be a Poisson manifold. Then we obtain a Lie algebroid $(T^\ast M,\pi^\sharp,[\cdot,\cdot]_\pi)$, where the anchor map $\pi^\sharp:T^\ast M\to TM$ is defined by $\pi^\sharp(\alpha)=\iota_\alpha\pi$ and the Lie bracket $[\cdot,\cdot]_\pi$ on $\Omega^1(M)$ is defined by
\[[\alpha,\beta]_\pi=\mathcal{L}_{\pi^\sharp(\alpha)}\beta-\mathcal{L}_{\pi^\sharp(\beta)}\alpha-d(\pi(\alpha,\beta))\](for more details, see \cite{crainic2021lectures}).\End
\end{example}

\vspace{0.2in}
For a Lie algebroid $(\mathcal{A},\rho,[\cdot,\cdot]_\mathcal{A})$, we can associate a cochain complex $(\Omega^k(\mathcal{A}),d_\mathcal{A})$ in the following way. We set $\Omega^k(\mathcal{A})=\Gamma(\wedge^k\mathcal{A}^\ast)$ and define the differential map $d_\mathcal{A}:\Omega^k(\mathcal{A})\to\Omega^{k+1}(\mathcal{A})$ as
\[\begin{split}
    d_\mathcal{A}\eta(v_0,\cdots,v_k)&=\displaystyle\sum_{i=0}^{k}(-1)^i\mathcal{L}_{\rho(v_i)}(\eta(v_0,\cdots,\hat{v_i},\cdots,v_k))\\
    &+\displaystyle\sum_{0\leq i<j\leq k}(-1)^{i+j}\eta([v_i,v_j]_\mathcal{A},v_0,\cdots,\hat{v_i},\cdots,\hat{v_j},\cdots,v_k)
\end{split}\]for $\eta\in\Omega^k(\mathcal{A})$ and $v_0,\cdots,v_k\in\Gamma(\mathcal{A})$. For $f\in C^\infty(M)$, we define $d_\mathcal{A}f\in\Omega^1(\mathcal{A})$ by $(d_\mathcal{A}f)(v)=\rho(v)f$ for $v\in\Gamma(\mathcal{A})$. Then we can prove that $d_\mathcal{A}\circ d_\mathcal{A}=0$ holds, thus $(\Omega^k(\mathcal{A}),d_\mathcal{A})$ is a cochain complex. We refer to it as the \textit{Lie algebroid complex}, and we refer to its cohomology $H^k(\mathcal{A}):=\ker d_\mathcal{A}/\im d_\mathcal{A}$ as the \textit{Lie algebroid cohomology} of $(\mathcal{A},\rho,[\cdot,\cdot]_\mathcal{A})$.

\begin{example}
  For the Lie algebroid $(TM,\operatorname{id}_{TM},[\cdot,\cdot])$, the Lie algebroid differential $d_{TM}$ is the exterior derivative, and its Lie algebroid cohomology $H^k(TM)$ coincides with the de~Rham cohomology $H^k_{\mathrm{dR}}(M)$.\End
\end{example}

\begin{example}\label{ex:1}
  In the case of $\mathfrak{g}\to\{{\rm pt}\}$, the Lie algebroid differential $d_\mathfrak{g}:\wedge^k\mathfrak{g}^\ast\to\wedge^{k+1}\mathfrak{g}^\ast$ is given by
  \[d_\mathcal{\mathfrak{g}}\eta(v_0,\cdots,v_k)=\displaystyle\sum_{0\leq i<j\leq k}(-1)^{i+j}\eta([v_i,v_j]_\mathfrak{g},v_0,\cdots,\hat{v_i},\cdots,\hat{v_j},\cdots,v_k),\]
  and its cohomology is the {\rm Chevalley--Eilenberg cohomology} of the Lie algebra $\mathfrak{g}$.\End
\end{example}

\begin{example}
  Let $\varphi:\mathfrak{g}\to\mathfrak{X}(M)$ be a Lie algebra action of a Lie algebra $\mathfrak{g}$ on a manifold $M$, and let $\mathcal{A}=M\ltimes\mathfrak{g}\to M$ be the associated action Lie algebroid. Then
\[
\Omega^k(\mathcal{A})
=
\Gamma(\Lambda^k\mathcal{A}^*)
\cong
C^\infty(M)\otimes\Lambda^k\mathfrak{g}^*
\]
holds. Moreover, the Lie algebroid differential $d_\mathcal{A}$ coincides with the Chevalley--Eilenberg differential associated with the $\mathfrak g$-module $C^\infty(M)$;
\begin{align*}
(d_\mathcal{A}\eta)(v_0,\ldots,v_k)
&=
\sum_i(-1)^i
\varphi(v_i)
\bigl(
\eta(v_0,\ldots,\hat v_i,\ldots,v_k)
\bigr)
\\
&\quad+
\sum_{i<j}
(-1)^{i+j}
\eta([v_i,v_j]_{\mathfrak g},
v_0,\ldots,\hat v_i,\ldots,\hat v_j,\ldots,v_k)
\end{align*}
for $\eta\in\Omega^k(\mathcal{A})$ and $v_0,\ldots,v_k\in\mathfrak g$. Therefore,
\[
H^\bullet(A)
\cong
H^\bullet(\mathfrak g;C^\infty(M)),
\]
that is, the Chevalley--Eilenberg cohomology of $\mathfrak g$ with coefficients in the $\mathfrak g$-module $C^\infty(M)$.

In particular, when $M=\{\ast\}$, the action Lie algebroid $M\ltimes\mathfrak g\to M$ reduces to the Lie algebroid $\mathfrak g\to\{\ast\}$, and its cohomology is the Chevalley--Eilenberg cohomology with respect to the trivial representation (\autoref{ex:1}).\End
\end{example}

\begin{example}
  For a Lie algebroid $(T^\ast M,\pi^\sharp,[\cdot,\cdot]_\pi)$ associated to a Poisson manifold $(M,\pi)$, the Lie algebroid complex $(\Omega^k(T^\ast M),d_{T^\ast M})$ turns out to be a cochain complex $(\mathfrak{X}^k(M),d_\pi)$, where the differential $d_\pi$ is defined by $d_\pi=[\pi,\cdot]_S$. The cohomology of the complex $(\mathfrak{X}^k(M),d_\pi)$ is called the {\rm Poisson cohomology} of $(M,\pi)$, which is precisely the Lie algebroid cohomology of $(T^\ast M,[\cdot,\cdot]_\pi,\pi^\sharp)$. \End
\end{example}

Let $(\mathcal{A}\to M,\rho_\mathcal{A},[\cdot,\cdot]_\mathcal{A})$ and $(\mathcal{B}\to N,\rho_\mathcal{B},[\cdot,\cdot]_\mathcal{B})$ be Lie algebroids. A bundle map $\Phi:\mathcal{B}\to\mathcal{A}$ is said to be a \textit{Lie algebroid morphism} when the pullback map $\Phi^!:\Omega^\bullet(\mathcal{A})\to\Omega^\bullet(\mathcal{B})$ is a cochain map, namely, $\Phi^! d_\mathcal{A}=d_\mathcal{B}\Phi^!$ holds. If $\Phi:\mathcal{B}\to\mathcal{A}$ is a Lie algebroid morphism, it gives rise to a map $\Phi^!:H^\bullet(\mathcal{A})\to H^\bullet(\mathcal{B})$.

\section{Godbillon--Vey classes for Lie subalgebroids}

In this section, we give the definition of the Godbillon--Vey classes for Lie subalgebroids. A Lie algebroid $(\mathcal{B}\to N,\rho_\mathcal{B},[\cdot,\cdot]_\mathcal{B})$ is said to be a \textit{Lie subalgebroid} of another Lie algebroid $(\mathcal{A}\to M,\rho_\mathcal{A},[\cdot,\cdot]_\mathcal{A})$ when $\mathcal{B}\to N$ is a subbundle of $\mathcal{A}\to M$ and the inclusion $\mathcal{B}\hookrightarrow\mathcal{A}$ is a Lie algebroid morphism. A subbundle $\mathcal{B}\subset\mathcal{A}$ over the same base manifold as $(\mathcal{A},\rho,[\cdot,\cdot]_\mathcal{A})$ is a Lie subalgebroid if and only if $[\Gamma(\mathcal{B}),\Gamma(\mathcal{B})]_{\mathcal{A}}\subset\Gamma(\mathcal{B})$ holds. Such a Lie subalgebroid is called \emph{wide}. For a wide Lie subalgebroid $\mathcal{B}\subset\mathcal{A}$, the pair $(\mathcal{A},\mathcal{B})$ is called a \emph{Lie pair} \cite{chen2016atiyah}. Unless otherwise stated, we consider only wide Lie subalgebroids in this paper.

Suppose that a corank $q$ Lie subalgebroid $\mathcal{B}\subset\mathcal{A}$ is \textit{coorientable}, i.e., the quotient bundle $\mathcal{A}/\mathcal{B}$ is orientable. Then there is a nowhere vanishing form $\alpha\in\Omega^q(\mathcal{A})$ which is locally decomposable and satisfies
\[\Gamma(\mathcal{B})=\{v\in\Gamma(\mathcal{A}) \mid \eta(v)=0,\ \eta\in\mathcal{I}\},\]where $\mathcal{I}$ is defined by
        \[\mathcal{I}=\{\eta\in\Omega^1(\mathcal{A})\mid \eta\wedge\alpha=0\}\]
(since $\mathcal{B}$ is coorientable, we can take $\alpha$ globally). If we decompose $\alpha=\alpha_1\wedge\cdots\wedge\alpha_q$ on an open subset $U\subset M$, then $\mathcal{I}$ is locally given on $U$ by
\[\mathcal{I}|_U=\left\{\displaystyle\sum_{i=1}^{q}f_i\alpha_i\ \middle|\ f_i\in C^\infty(U)\right\}\]
and we have
\[
\mathcal{B}|_U=\bigcap_{i=1}^q \ker \alpha_i.
\]
Then the condition $[\Gamma(\mathcal{B}),\Gamma(\mathcal{B})]_\mathcal{A} \subset \Gamma(\mathcal{B})$ can be expressed in the form $d_\mathcal{A}\alpha = \beta \wedge \alpha$ for some $\beta \in \Omega^1(\mathcal{A})$. 

\begin{theorem}\label{thm:0}
  The cohomology class $[\beta\wedge (d_\mathcal{A}\beta)^q]_\mathcal{A}\in H^{2q+1}(\mathcal{A})$ is independent of the choice of the forms $\alpha$ and $\beta$, and thus uniquely determined by the Lie subalgebroid $\mathcal{B}$.
\end{theorem}
\begin{proof}
Applying $d_{\mathcal A}$ to both sides of $d_{\mathcal A}\alpha=\beta\wedge\alpha,
$ we obtain
$$
0
=d_{\mathcal A}\beta\wedge\alpha
-\beta\wedge d_{\mathcal A}\alpha
=d_{\mathcal A}\beta\wedge\alpha.
$$
Hence, by (a generalized version of) Cartan's lemma, $d_{\mathcal{A}}\beta\in \mathcal{I}\wedge\Omega^1(\mathcal{A})$ holds, where $\mathcal{I}\wedge\Omega^1(\mathcal{A})$ is an ideal locally given by
\[(\mathcal{I}\wedge\Omega^1(\mathcal{A}))|_U=\left\{\displaystyle\sum_{i=1}^{q}\alpha_i\wedge\eta_i\ \middle|\ \eta_i\in \Omega^1(\mathcal{A}|_U)\right\}.\]
This implies that $(d_{\mathcal A}\beta)^{q+1}=0$, and thus $\beta\wedge (d_{\mathcal A}\beta)^q$ is $d_{\mathcal A}$-closed.

\vspace{0.1in}
Suppose that $\beta'\in\Omega^1(\mathcal A)$ is another form satisfying $d_{\mathcal A}\alpha=\beta'\wedge\alpha.$ Then
$$
(\beta'-\beta)\wedge\alpha
=
d_{\mathcal A}\alpha-d_{\mathcal A}\alpha
=
0
$$
holds. Hence $\gamma:=\beta'-\beta$ satisfies $\gamma\in\mathcal{I}$. In addition, we have $d_{\mathcal A}\beta,d_{\mathcal A}\beta'\in\mathcal{I}\wedge\Omega^1(\mathcal{A})$, and thus we obtain
\begin{equation}\label{eq:1}
  \gamma\wedge(d_{\mathcal A}\beta')^q=0,
\end{equation}
\begin{equation}\label{eq:2}
  \gamma\wedge d_{\mathcal A}\beta\wedge
(d_{\mathcal A}\beta')^k\wedge(d_{\mathcal A}\beta)^{q-1-k}=0.
\end{equation}By (\ref{eq:1}),
\[
\begin{aligned}
&\beta'\wedge(d_{\mathcal A}\beta')^q
-\beta\wedge(d_{\mathcal A}\beta)^q\\
&=(\beta+\gamma)\wedge(d_{\mathcal A}\beta')^q
-\beta\wedge(d_{\mathcal A}\beta)^q\\
&=\beta\wedge\left\{
(d_{\mathcal A}\beta')^q-(d_{\mathcal A}\beta)^q
\right\}\\
&=\beta\wedge d_{\mathcal A}\gamma\wedge S
\end{aligned}
\]holds, where
\[
S:=\sum_{k=0}^{q-1}
(d_{\mathcal A}\beta')^k\wedge(d_{\mathcal A}\beta)^{q-1-k}.
\]
On the other hand, by (\ref{eq:2}) and $d_{\mathcal A}S=0$, we have
\[
d_{\mathcal A}(\beta\wedge\gamma\wedge S)
=d_{\mathcal A}\beta\wedge\gamma\wedge S-\beta\wedge d_{\mathcal A}\gamma\wedge S=-\beta\wedge d_{\mathcal A}\gamma\wedge S.
\]
Therefore we obtain
\[
\beta'\wedge(d_{\mathcal A}\beta')^q
-\beta\wedge(d_{\mathcal A}\beta)^q=
-d_{\mathcal A}\left(
\beta\wedge\gamma\wedge S 
\right).
\]

Finally, let $\alpha'\in\Omega^q(\mathcal{A})$ be another form defining $\mathcal B$. Then we have $\alpha'=g\alpha$ for some nowhere-vanishing function $g\in C^\infty(M)$. Hence,
$$
d_{\mathcal A}\alpha'
=d_{\mathcal A}g\wedge\alpha
+gd_{\mathcal A}\alpha
=\left(\frac{d_{\mathcal A}g}{g}+\beta\right)\wedge\alpha'
$$
holds. Thus, setting
$$
\beta'=\frac{d_{\mathcal A}g}{g}+\beta,
$$
we have $
d_{\mathcal A}\alpha'=\beta'\wedge\alpha'.
$ Moreover, for this choice of $\beta'$,
$$
\begin{aligned}
\beta'\wedge (d_{\mathcal A}\beta')^q
&=
\left(\frac{d_{\mathcal A}g}{g}+\beta\right)\wedge (d_{\mathcal A}\beta)^q\\
&=
\beta\wedge (d_{\mathcal A}\beta)^q
+d_{\mathcal A}\log|g|\wedge(d_{\mathcal A}\beta)^q\\
&=\beta\wedge (d_{\mathcal A}\beta)^q
+d_{\mathcal A}\left\{
\log|g|\,(d_{\mathcal A}\beta)^q
\right\}.
\end{aligned}
$$
Therefore, the class $[\beta\wedge (d_{\mathcal A}\beta)^q]_{\mathcal A}$ is independent of the choices of $\alpha$ and $\beta$.
\end{proof}

\begin{definition}
        We call the cohomology class $\gv_{\mathcal{A}}\mathcal{B}:=[\beta\wedge (d_\mathcal{A}\beta)^q]_\mathcal{A}\in H^{2q+1}(\mathcal{A})$ the \textit{Godbillon--Vey class} of $\mathcal{B}$ with respect to $\mathcal{A}$.\End
\end{definition}

Let $(\mathcal{A},\rho,[\cdot,\cdot]_\mathcal{A})$ be a Lie algebroid over $M$ and $f:N\to M$ a smooth map. Assume that $f$ is transverse to $\mathcal{A}$, i.e., $f_\ast(TN)+\rho(\mathcal{A})=TM$ holds. Then
\[f^!\mathcal{A}:=\mathcal{A}\times_{TM}TN=\{(v,X)\in\mathcal{A}\times TN\mid\rho(v)=f_\ast X\}\]
is a vector bundle over $N$ (subbundle of $\mathcal{A}\times_{M}TN\to N$). Moreover, we obtain a Lie algebroid $(f^!\mathcal{A},\rho_{f^!\mathcal{A}},[\cdot,\cdot]_{f^!\mathcal{A}})$ where the anchor map $\rho_{f^!\mathcal{A}}$ is the projection on $TN$ and the Lie bracket $[\cdot,\cdot]_{f^!\mathcal{A}}$ is defined by 
\[[(f^\ast v_1,X_1),(f^\ast v_2,X_2)]_{f^!\mathcal{A}}:=(f^\ast[v_1,v_2]_{\mathcal{A}},[X_1,X_2])\]
for sections of the form $(f^\ast v,X)$ ($v\in\Gamma(\mathcal{A})$ and $X\in\mathfrak{X}(N)$), and then is extended to $\Gamma(f^!\mathcal{A})\times\Gamma(f^!\mathcal{A})$ uniquely by requiring the Leibniz rule. The Lie algebroid $f^!\mathcal{A}$ is called the \textit{pullback Lie algebroid} of $(\mathcal{A},\rho,[\cdot,\cdot]_\mathcal{A})$ by $f$.

\begin{remark}
  In this paper, we consider only the Godbillon--Vey classes of wide Lie subalgebroids. For a Lie subalgebroid $\mathcal{B}\to N$ whose base manifold differs from $M$, assume that the pullback Lie algebroid $i^!\mathcal{A}\to N$ along the inclusion map $i:N\hookrightarrow M$ exists and that the wide Lie subalgebroid $\mathcal{B}\subset i^!\mathcal{A}$ is coorientable. Then its Godbillon--Vey class can be defined using the construction in this paper.\End
\end{remark}

By using the projection $f^!\mathcal{A}\to\mathcal{A}$, one can pullback a section of $\Omega^k(\mathcal{A})$ to a section of $\Omega^k(f^!\mathcal{A})$. We denote the map as $f^!:\Omega^\bullet(\mathcal{A})\to\Omega^\bullet(f^!\mathcal{A})$.

\begin{proposition}\label{prop:1}
  The projection $f^!\mathcal{A}\to\mathcal{A}$ is a Lie algebroid morphism, namely, the map $f^!:\Omega^\bullet(\mathcal{A})\to\Omega^\bullet(f^!\mathcal{A})$ commutes with differential maps $d_\mathcal{A},d_{f^!\mathcal{A}}$. Consequently, it gives rise to a map $f^!:H^\bullet(\mathcal{A})\to H^\bullet(f^!\mathcal{A})$.
\end{proposition}
\begin{proof}
Choose open subsets $U\subset N$ and $V\subset M$ such that $f(U)\subset V$, and let $(e_1,\ldots,e_r)$ be a local frame for \(\mathcal A|_V\). Then any section \(u\in\Gamma((f^!\mathcal{A})|_U)\) can be written as
\[
u=\left(\sum_i h_i f^*e_i,X\right),
\]
where $X\in\mathfrak{X}(U)$ and \(h_i\in C^\infty(U)\) satisfy
\[
f_\ast X=\sum_i h_i(\rho(e_i)\circ f).
\]
Hence, for $F\in C^\infty(V)$,
\[
(d_{f^!\mathcal{A}}f^! F)(u)
=X(F\circ f)
=\sum_i h_i(\rho(e_i)F)\circ f
=(f^!d_{\mathcal A}F)(u)
\]
holds. Next, if two local sections $u_1,u_2\in\Gamma((f^!\mathcal{A})|_U)$ are written as
\[
u_1=\left(\sum_i h_i f^*e_i,X\right),\qquad
u_2=\left(\sum_j g_j f^*e_j,Y\right),
\]
then their Lie bracket is given by
\[
[u_1,u_2]_{f^!\mathcal{A}}=
\left(
\sum_{i,j}h_ig_jf^*[e_i,e_j]_{\mathcal A}
+\sum_jX(g_j)f^*e_j
-\sum_iY(h_i)f^*e_i,\ [X,Y]
\right).
\]
For $\theta\in\Omega^1(\mathcal A|_V)$, we have
\[
\begin{aligned}
(d_{f^!\mathcal{A}}f^!\theta)(u_1,u_2)
&=X\left(\sum_j g_j\,\theta(e_j)\circ f\right)
-Y\left(\sum_i h_i\,\theta(e_i)\circ f\right)
-(f^!\theta)([u_1,u_2]_{f^!\mathcal{A}}).
\end{aligned}
\]
Expanding the right-hand side, the terms involving $X(g_j)$ and $Y(h_i)$ cancel with the corresponding terms from the Lie bracket, yielding
\[
\begin{aligned}
(d_{f^!\mathcal{A}}f^!\theta)(u_1,u_2)
&=\sum_{i,j}h_ig_j
\left(
\rho(e_i)(\theta(e_j))
-\rho(e_j)(\theta(e_i))
-\theta([e_i,e_j]_{\mathcal A})
\right)\circ f\\
&=\sum_{i,j}h_ig_j
\bigl((d_{\mathcal A}\theta)(e_i,e_j)\circ f\bigr)\\
&=(f^!d_{\mathcal A}\theta)(u_1,u_2).
\end{aligned}
\]
Finally, since the differentials satisfy the graded Leibniz rule and $f^!$ preserves wedge products, the identity
\[
d_{f^!\mathcal{A}}f^!\eta=f^!d_{\mathcal A}\eta
\]
holds for every form $\eta$, since every form can locally be expressed as a sum of products of functions and $1$-forms.

\end{proof}

We now show the naturality of the Godbillon--Vey class defined above. Let $\mathcal{B}\subset\mathcal{A}$ be a Lie subalgebroid of corank $q$ which is transverse to $f$, i.e., $f_\ast(TN)+\rho(\mathcal{B})=TM$ holds. Then the pullback Lie algebroids $f^!\mathcal{A}$ and $f^!\mathcal{B}$ are defined, and $f^!\mathcal{B}$ is a Lie subalgebroid of $f^!\mathcal{A}$ of corank $q$. Moreover, defining $\pi:f^!\mathcal{A}\to f^\ast(\mathcal{A}/\mathcal{B})$ by $\pi(v,X)=\bigl(x,[v]\bigr)$\ $(X\in T_xN)
$, we obtain a short exact sequence of vector bundles
\[
0\longrightarrow f^!B\longrightarrow f^!A
\xrightarrow{\;\pi\;}f^*(A/B)\longrightarrow0
\]and hence an isomorphism of vector bundles $f^!A/f^!B\cong f^*(A/B)$. Therefore, if $\mathcal{B}\subset\mathcal{A}$ is coorientable, then so is $f^!\mathcal{B}\subset f^!\mathcal{A}$.

\begin{theorem}\label{thm:100}
  Under the above settings,
  \[\gv_{f^!\mathcal{A}}(f^!\mathcal{B})=f^!\gv_\mathcal{A}(\mathcal{B})\]
  holds.
\end{theorem}
\begin{proof}
  Take a form $\alpha\in\Omega^q(\mathcal{A})$ which defines $\mathcal{B}$ and locally decompose it as $\alpha=\alpha_1\wedge\cdots\wedge\alpha_q$. By the transversality condition, $f^!\alpha_1,\cdots,f^!\alpha_q$ are linearly independent. Hence $f^!\alpha=f^!\alpha_1\wedge\cdots\wedge f^!\alpha_q$ is nowhere vanishing and we have 
  \[(f^!\mathcal{B})|_U=\bigcap_{i=1}^q \ker (f^!\alpha_i)\]
  on an open subset $U\subset N$.
  
  By \autoref{prop:1}, if $d_\mathcal{A}\alpha=\beta\wedge\alpha$, then $d_{f^!\mathcal{A}}(f^!\alpha)=f^!\beta\wedge f^!\alpha$ holds. Moreover, \autoref{prop:1} also implies that $(f^!\beta)\wedge (d_{f^!\mathcal{A}}(f^!\beta))^q=f^!(\beta\wedge (d_\mathcal{A}\beta)^q)$, and this completes the proof.
\end{proof}

\begin{remark}
  In the case $\mathcal{A}=TM$, the Godbillon--Vey class defined in this paper coincides with the classical Godbillon--Vey class of a foliation, and \autoref{thm:100} is nothing but the naturality of the classical Godbillon--Vey class.\End
\end{remark}

As a consequence of \autoref{thm:100}, we obtain the following invariance property of the Godbillon--Vey class defined above:
\begin{corollary}
  Let $\mathcal{A}$ be a Lie algebroid over $M$, and let $\mathcal{B}_1,\mathcal{B}_2\subset\mathcal{A}$ be Lie subalgebroids. Suppose that there exist a diffeomorphism $f:M\to M$ and a Lie algebroid automorphism $F:\mathcal{A}\to\mathcal{A}$ covering $f$ such that $F(\mathcal{B}_1)=\mathcal{B}_2$. Then
  \[\gv_\mathcal{A}(\mathcal{B}_1)=F^!\gv_\mathcal{A}(\mathcal{B}_2)\]
  holds.
\end{corollary}
\begin{proof}
  Let $\rho$ denote the anchor map of $\mathcal{A}$, and define a Lie algebroid isomorphism $\Phi:\mathcal{A}\to f^!\mathcal{A}$ by $v\mapsto(F(v),\rho(v))$. Since $F(\mathcal{B}_1)=\mathcal{B}_2$, we have $\Phi(\mathcal{B}_1)=f^!\mathcal{B}_2$. Thus, by \autoref{thm:100},
  \[\gv_\mathcal{A}(\mathcal{B}_1)=\Phi^!\gv_{f^!\mathcal{A}}(f^!\mathcal{B}_2)=\Phi^!f^!\gv_\mathcal{A}(\mathcal{B}_2)=F^!\gv_\mathcal{A}(\mathcal{B}_2)\]
  holds.
\end{proof}

\section{Examples}\label{examples}

In this section, we present several examples for which the Godbillon--Vey class is nontrivial.

\subsection{Lie algebras}\label{subsection:4.1}

Let $\mathfrak{h}\subset\mathfrak{g}$ be a codimension one Lie subalgebra. We consider a Lie subalgebroid $\mathfrak{h}\to\{\ast\}$ of $\mathfrak{g}\to\{\ast\}$. The Lie algebroid cohomology reduces to the Chevalley--Eilenberg cohomology (\autoref{ex:1}). Assume that $\mathfrak{h}$ has codimension one, and choose $\alpha \in \mathfrak{g}^*$ such that $\ker \alpha = \mathfrak{h}$. 

Choose $e\in\mathfrak g\setminus\mathfrak h$, so that $\mathfrak g=\mathfrak h\oplus\mathbb{R}e$. Then one can take $\beta \in \mathfrak{g}^*$ such that $d_{\mathfrak{g}}\alpha = \beta \wedge \alpha$ as 
\[\beta(v) = -\frac{\alpha([v,e])}{\alpha(e)}.\]
Indeed, write
$v\in\mathfrak g$ as
\[
v=u+ke \qquad (u\in\mathfrak h,\; k\in\mathbb R).
\]
Then
\[
\begin{split}
d_{\mathfrak g}\alpha(v_0,v_1)
&=-\alpha([u_0+k_0e,u_1+k_1e])\\
&=-\alpha(k_1[u_0,e]+k_0[e,u_1]).
\end{split}
\]
On the other hand,
\[
\begin{split}
(\beta\wedge\alpha)(v_0,v_1)
&=\bigl(\beta(v_0)\alpha(v_1)-\beta(v_1)\alpha(v_0)\bigr)\\
&=-\frac{\alpha([v_0,e])}{\alpha(e)}k_1\alpha(e)
+\frac{\alpha([v_1,e])}{\alpha(e)}k_0\alpha(e)\\
&=-\alpha([u_0,e])k_1+\alpha([u_1,e])k_0.
\end{split}
\]
Since
$$
d_{\mathfrak g}\beta(v_0,v_1)
=
-\beta([v_0,v_1])
=
\frac{\alpha([[v_0,v_1],e])}{\alpha(e)},
$$
we obtain

\[
\begin{split}
(\beta\wedge d_{\mathfrak g}\beta)(v_0,v_1,v_2)
&=
-\bigl(
\beta(v_0)\beta([v_1,v_2])
+\text{cyclic permutations}
\bigr)\\
&=
-\left(
\frac{\alpha([v_0,e])}{\alpha(e)}
\cdot
\frac{\alpha([[v_1,v_2],e])}{\alpha(e)}
+\text{cyclic permutations}
\right)\\
&=
-\frac{1}{\alpha(e)^2}
\bigl(
\alpha([v_0,e])\alpha([[v_1,v_2],e])
+\text{cyclic permutations}
\bigr).
\end{split}
\]

\vspace{0.1in}
For instance, let
$\mathfrak g=\mathfrak{sl}_2\mathbb R$
and let
$(H,X,Y)$
be a basis satisfying
\[
[H,X]=2X,\qquad
[H,Y]=-2Y,\qquad
[X,Y]=H.
\]
Consider the Borel subalgebra
$\mathfrak h=\operatorname{span}\langle H,X\rangle$.
Then
$\mathfrak{sl}_2\mathbb R/\mathfrak h
\simeq
\operatorname{span}\langle Y\rangle$.
Choose
$\alpha\in(\mathfrak{sl}_2\mathbb R)^*$
such that
\[
\alpha(H)=\alpha(X)=0,
\qquad
\alpha(Y)=1.
\]
Writing
$v_i=h_iH+x_iX+y_iY$
$(i=0,1,2)$,
we have
\[\begin{split} 
  \alpha[v_0,Y]\alpha[[v_1,v_2],Y]&=\alpha(-2h_0Y+x_0 H)\alpha([h_1x_2[H,X]+x_1h_2[X,H]\\
  &\qquad+y_1[Y,h_2H+x_2X]+y_2[h_1H+x_1X,Y],Y])\\ 
  &=-2h_0\alpha([2(h_1x_2-x_1h_2)X+y_1(2h_2Y-x_2H)+y_2(-2h_1Y+x_1H),Y])\\ 
  &=-2h_0\alpha([(x_1y_2-x_2y_1)H,Y])=-2h_0(-2(x_1y_2-x_2y_1))\\ 
  &=4h_0(x_1y_2-x_2y_1).
\end{split}\]
Hence, we obtain
\[
(\beta\wedge d_{\mathfrak{sl}_2\mathbb R}\beta)
(v_0,v_1,v_2)
=
-4
\bigl(
h_0(x_1y_2-x_2y_1)
+\text{cyclic permutations}
\bigr).
\]
Since $H^3(\mathfrak{sl}_2\mathbb R)\cong\mathbb R$ and $\wedge^3(\mathfrak{sl}_2\mathbb R)^\ast$ is one-dimensional, every nonzero $3$-form represents a nontrivial cohomology class. Hence the Godbillon--Vey class of
$\mathfrak h\subset\mathfrak{sl}_2\mathbb R$
is nontrivial.

\subsection{Action Lie algebroids}\label{Action Lie algebroids}

Let $\varphi:\mathfrak{g}\to\mathfrak{X}(M)$ be a Lie algebra action on a manifold $M$, and let $\mathcal{A}=M\ltimes\mathfrak{g}\to M$ be the corresponding action Lie algebroid. Let $\mathfrak{h}\subset\mathfrak{g}$ be a Lie subalgebra, and set
\[
  \mathcal{B}=M\ltimes\mathfrak{h}\subset \mathcal{A}.
\]
Then $\mathcal{B}$ is a Lie subalgebroid of $\mathcal{A}$ and the quotient $\mathcal{A}/\mathcal{B}$ is trivial. Assume that $x\in M$ is a fixed point of the action, that is,
\[
  \varphi(v)_x=0
  \qquad
  \text{for all }v\in\mathfrak{g}.
\]

\begin{proposition}
  The evaluation map at $x$
\[
  \operatorname{ev}_x:
  \Omega^\bullet(\mathcal{A})
  =
  C^\infty(M)\otimes \Lambda^\bullet\mathfrak{g}^*
  \longrightarrow
  \Lambda^\bullet\mathfrak{g}^*
\]
defined by $\operatorname{ev}_x(\eta)=\eta(x)$ is a cochain map.
\end{proposition}
\begin{proof}
  For $\eta\in\Omega^k(\mathcal{A})$ and
$v_0,\ldots,v_k\in\mathfrak{g}$, we compute using the constant sections
of $\mathcal{A}$ determined by $v_0,\ldots,v_k$.  The Lie algebroid
differential of the action Lie algebroid is given by
\[
\begin{aligned}
(d_{\mathcal{A}}\eta)(v_0,\ldots,v_k)
&=
\sum_{i=0}^k (-1)^i
\varphi(v_i)
\bigl(
\eta(v_0,\ldots,\widehat{v_i},\ldots,v_k)
\bigr)
\\
&\quad+
\sum_{0\leq i<j\leq k}
(-1)^{i+j}
\eta([v_i,v_j]_{\mathfrak{g}},
v_0,\ldots,\widehat{v_i},\ldots,\widehat{v_j},\ldots,v_k).
\end{aligned}
\]
Evaluating this at the fixed point $x$, the first summation vanishes because
\[
  \varphi(v_i)_x=0
  \qquad
  (i=0,\ldots,k).
\]
Hence
\[
\begin{aligned}
\operatorname{ev}_x(d_{\mathcal{A}}\eta)(v_0,\ldots,v_k)
&=
\sum_{0\leq i<j\leq k}
(-1)^{i+j}
\eta(x)([v_i,v_j]_{\mathfrak{g}},
v_0,\ldots,\widehat{v_i},\ldots,\widehat{v_j},\ldots,v_k)
\\
&=
(d_{\mathfrak{g}}\operatorname{ev}_x(\eta))(v_0,\ldots,v_k).
\end{aligned}
\]
\end{proof}

Therefore $\operatorname{ev}_x$ induces a homomorphism
\[
  \operatorname{ev}_x:
  H^\bullet(\mathcal{A})
  \longrightarrow
  H^\bullet(\mathfrak{g}),
\]
where the right-hand side is the Chevalley--Eilenberg cohomology of
$\mathfrak{g}$. Moreover, the evaluation map preserves wedge products:
\[
  \operatorname{ev}_x(\eta_1\wedge\eta_2)
  =
  \operatorname{ev}_x(\eta_1)\wedge \operatorname{ev}_x(\eta_2).
\]
Suppose that $\mathcal{B}\subset\mathcal{A}$ has corank $q$ and let $\alpha\in\Omega^q(\mathcal{A})$ be a defining form of $\mathcal{B}$. Take $\beta\in\Omega^1(\mathcal{A})$ such that $d_{\mathcal{A}}\alpha=\beta\wedge\alpha$. Evaluating this equation at the fixed point $x$, we obtain
\[
  d_{\mathfrak{g}}\operatorname{ev}_x(\alpha)
  =
  \operatorname{ev}_x(\beta)\wedge \operatorname{ev}_x(\alpha).
\]
Thus $\operatorname{ev}_x(\alpha)$ is a defining form for
$\mathfrak{h}\subset\mathfrak{g}$, and $\operatorname{ev}_x(\beta)$ is a
corresponding $1$-form in the definition of the Godbillon--Vey class for
the Lie subalgebroid $(\mathfrak{h}\to\{\ast\})\subset(\mathfrak{g}\to\{\ast\})$. Then it follows that
\[
\begin{aligned}
\operatorname{ev}_x
\left(
  \beta\wedge (d_{\mathcal{A}}\beta)^q
\right)
&=
\operatorname{ev}_x(\beta)
\wedge
\left(
  d_{\mathfrak{g}}\operatorname{ev}_x(\beta)
\right)^q.
\end{aligned}
\]
Consequently, we obtain
\[
  \operatorname{ev}_x
  \bigl(
    \operatorname{GV}_{\mathcal{A}}(\mathcal{B})
  \bigr)
  =
  \operatorname{GV}_{\mathfrak{g}}(\mathfrak{h})
  \in H^{2q+1}(\mathfrak{g}).
\]

\vspace{0.2in}
Let $\mathfrak{h}\subset\mathfrak{sl}_2\mathbb{R}$ be as in \autoref{subsection:4.1}. Consider the Lie algebra action of $\mathfrak{sl}_2\mathbb{R}$ on $M=\mathbb{R}^2$ induced by the standard left action of $SL_2\mathbb{R}$, with the convention
\[
  \varphi(A)_{(x,y)}
  =\left.\frac{d}{dt}\right|_{t=0}\exp(-tA)(x,y)^T.
\]
This convention makes $\varphi$ a Lie algebra homomorphism. This action defines the action Lie algebroid
\[
  \mathcal{A}=\mathbb{R}^2\ltimes\mathfrak{sl}_2\mathbb{R}
  \longrightarrow \mathbb{R}^2.
\]
Under the canonical identification $T_{(x,y)}\mathbb{R}^2\cong\mathbb{R}^2$, the anchor is given by
\[
  \rho((x,y),A)=-A(x,y)^T.
\]
Since $\mathfrak{h}\subset\mathfrak{sl}_2\mathbb{R}$ is a Lie subalgebra,
\[
  \mathcal{B}=\mathbb{R}^2\ltimes\mathfrak{h}
  \subset \mathcal{A}
\]
is a Lie subalgebroid. The origin $\boldsymbol{0}=(0,0)\in\mathbb{R}^2$ is a fixed point of the standard action. Hence the evaluation map
\[
  \operatorname{ev}_{\boldsymbol{0}}:\Omega^\bullet(\mathcal{A})
  =
  C^\infty(\mathbb{R}^2)\otimes\Lambda^\bullet(\mathfrak{sl}_2\mathbb{R})^*
  \longrightarrow
  \Lambda^\bullet(\mathfrak{sl}_2\mathbb{R})^*
\]
is a cochain map. Therefore we have 
\[
  \operatorname{ev}_{\boldsymbol{0}}\bigl(\operatorname{GV}_{\mathcal{A}}(\mathcal{B})\bigr)
  =
  \operatorname{GV}_{\mathfrak{sl}_2\mathbb{R}}(\mathfrak{h}).
\]
Since $\operatorname{GV}_{\mathfrak{sl}_2\mathbb{R}}(\mathfrak{h})\neq0$, we obtain $\operatorname{GV}_{\mathcal{A}}(\mathcal{B})\neq0$.

\vspace{0.1in}
The anchor image of $\mathcal{B}$ does not have constant rank. Indeed, for $(x,y)\in\mathbb{R}^2$, we have
\[
  H(x,y)^T=(x,-y)^T,
  \qquad
  X(x,y)^T=(y,0)^T.
\]
Hence
\[
  \rho(\mathcal{B})_{(x,y)}
  =
  \operatorname{span}\{(x,-y),(y,0)\}.
\]
holds. Therefore we obtain
\[
  \dim \rho(\mathcal{B})_{(x,y)}
  =
  \begin{cases}
    2 & \quad\textrm{if}\ \ y\neq0,\\
    1 & \quad\textrm{if}\ \ y=0,\ x\neq0,\\
    0 & \quad\textrm{if}\ \ (x,y)=(0,0),
  \end{cases}
\]and thus $\mathcal{B}\subset \mathcal{A}$ has a nontrivial Godbillon--Vey class, while its anchor image defines a singular foliation.

\subsection{The cotangent bundle of the Poisson manifold $\mathfrak{g}^\ast$}

In general, the dual space $\mathfrak{g}^\ast$ of a Lie algebra $\mathfrak{g}$ carries a natural Poisson structure, called the \textit{linear Poisson structure}. For $u\in\mathfrak{g}$, let $\ell_u(\xi)=\xi(u)$ be the corresponding linear function on $\mathfrak{g}^\ast$. We use the convention
\[
\{\ell_u,\ell_v\}=\ell_{[u,v]},\qquad
\varphi(u)=\pi^\sharp(d\ell_u),
\]
where $\pi$ denotes the linear Poisson bivector. Thus, under the canonical identification $T_\xi\mathfrak{g}^\ast\cong\mathfrak{g}^\ast$, we have
\[
\varphi(u)_\xi(v)=\xi([u,v]),\qquad u,v\in\mathfrak{g}.
\]
With the coadjoint representation defined by $(\operatorname{ad}_u^*\xi)(v)=-\xi([u,v])$, this means that $\varphi(u)_\xi=-\operatorname{ad}_u^*\xi$. The map $\varphi:\mathfrak{g}\to\mathfrak{X}(\mathfrak{g}^\ast)$ is a Lie algebra action. Under the identification $(\xi,u)\mapsto(d\ell_u)_\xi$, the action Lie algebroid $\mathfrak{g}^\ast\ltimes\mathfrak{g}$ is isomorphic to the cotangent Lie algebroid $T^\ast\mathfrak{g}^\ast$ of this Poisson manifold (\autoref{ex:0}).

\vspace{0.1in}
Consider the case $\mathfrak g=\mathfrak{sl}_2\mathbb{R}$. Let $\mathcal{A}=(\mathfrak{sl}_2\mathbb{R})^\ast \ltimes\mathfrak{sl}_2\mathbb{R}$ be the action Lie algebroid and $\mathcal{B}=(\mathfrak{sl}_2\mathbb{R})^\ast \ltimes\mathfrak{h}$ its Lie subalgebroid associated with the action by the subalgebra $\mathfrak{h}\subset\mathfrak{sl}_2\mathbb{R}$ as in \autoref{subsection:4.1}. Since $\boldsymbol{0}\in(\mathfrak{sl}_2\mathbb{R})^\ast$ is a fixed point of the coadjoint action, the same argument as in \autoref{Action Lie algebroids} shows that $\operatorname{GV}_{\mathcal{A}}(\mathcal{B})\neq0$.

Moreover, if we define coordinates $(a,b,c)$ on $(\mathfrak{sl}_2\mathbb{R})^\ast$ by
\[
a=\xi(H),\qquad b=\xi(X),\qquad c=\xi(Y)
\]
for $\xi\in(\mathfrak{sl}_2\mathbb{R})^\ast$, then the bracket relations in \autoref{subsection:4.1} give the following expressions for the anchor applied to the constant sections $H$ and $X$:
\[
\rho(H)=2b\displaystyle\frac{\partial}{\partial b}-2c\displaystyle\frac{\partial}{\partial c},\qquad
\rho(X)=-2b\displaystyle\frac{\partial}{\partial a}+a\displaystyle\frac{\partial}{\partial c}.
\]
Since $\mathfrak{h}=\operatorname{span}\langle H,X\rangle$, we obtain
\[
\rho(\mathcal{B})_{(a,b,c)}
=\operatorname{span}\{(0,2b,-2c),\,(-2b,0,a)\}
\]
under the coordinate identification $T_{(a,b,c)}(\mathfrak{sl}_2\mathbb{R})^\ast\cong\mathbb{R}^3$. Therefore,
\[
\dim \rho(\mathcal{B})_{(a,b,c)}=
\begin{cases}
2,& b\neq0,\\
1,& b=0,\ (a,c)\neq(0,0),\\
0,& (a,b,c)=(0,0,0).
\end{cases}
\]
Thus, the singular foliation induced by $\mathcal{B}$ is not regular, while $\operatorname{GV}_{\mathcal{A}}(\mathcal{B})$ is nonzero.

\section{Further study}

As mentioned in the introduction, the Godbillon--Vey class is a special case of secondary characteristic classes \cite{bott1972lectures}. A natural direction for future research is to extend the theory of secondary characteristic classes to Lie subalgebroids.

\vspace{0.1in}
A geometric structure closely related to the Poisson structure is the so-called \textit{Jacobi structure}. A Jacobi structure on a manifold is defined by a pair of a vector field $E$ and a bivector field $\pi$ such that 
\[[\pi,\pi]_S=2E\wedge\pi,\qquad [\pi,E]_S=0.\]
The notion of a Jacobi manifold is a natural generalization of that of a Poisson manifold. It is known that, just as a Poisson manifold induces a Lie algebroid, a Jacobi manifold also induces a Lie algebroid (see \cite{vaisman2000bv}). The singular foliation determined by this Lie algebroid consists of odd-dimensional leaves equipped with contact structures and even-dimensional leaves equipped with locally conformal symplectic structures. In \cite{yonehara2024godbillon}, an explicit expression of Godbillon--Vey classes was obtained for regular foliations induced by Jacobi manifolds. It would also be interesting to investigate the relationship between the classical Godbillon--Vey class and the Godbillon--Vey class introduced in this paper for Poisson and Jacobi manifolds with regular foliations.

\vspace{0.2in}
\bibliography{hoge} 
\bibliographystyle{alpha}

\end{document}